\documentclass[conference,a4paper,11pt]{article}
\usepackage[latin1]{inputenc}
\usepackage[T1]{fontenc}
\usepackage[left=3cm,right=2.5cm,top=2.5cm,bottom=3cm]{geometry}
\usepackage{color}
\usepackage{bm}
\usepackage{amsmath}
\usepackage{lmodern}
\usepackage{amsthm}
\usepackage[mathcal]{eucal}
\usepackage{amssymb}
\usepackage{mathrsfs}
\begin{document}
\title{\textbf{Reliability of three-state $k$-out-of-$n:G$ system with non-homogeneous Markov dependent components}}
\author{Abdelmoumene Boulahia$^{1}$ and Soheir Belaloui$^{2}$
}
\date{Departement of Mathematics, Faculty of exact sciences, University of fr\`eres Mentouri Constantine 1, Algeria\\
Laboratory of Mathematical Modeling and Simulation\\
$^{1}$ abdelmoumen.boulahia@doc.umc.edu.dz\\
$^{2}$ belaloui.soheir@umc.edu.dz
}
\maketitle
\vspace{1 cm}
\begin{abstract}{}
\noindent
In this paper, we study the reliability of a three-state $k$-out-of-$n:G$ system. We consider the situation where the system components are non-homogeneous Markov dependent, and we derive a closed-form formula for the system reliability, including increasing three-state $k$-out-of-$n:G$ system and decreasing three-state $k$-out-of-$n:G$ system. Our study is based on the probability generating function method. Two numerical examples are presented to demonstrate the use of the formula. 
\end{abstract}
\begin{center}
\textbf{NOTATION} 
\end{center}
\begin{tabular}{ll}
$n$       & number of components in the system. \\
$X_u$     & state of component $u$, $X_u\in\left\{0,1,2\right\}$, $u=1,2,...,n$.\\
$R^j(n)$  & probability that the three-state $k$-out-of-$n:G$ system is in state $j$ or above, $j=1,2$.\\
$r^j(n)$  & probability that the three-state $k$-out-of-$n:G$ system is in state $j$, $j=0,1,2$.\\
$\mathcal{I}_a(b)$& indicator function of a number $a$ ($\mathcal{I}_a(b)=1$ if $a\leq b$ and $0$ otherwise).\\
$N_{n,j}$ & the total number of components that are in state $j$ or above ($j=1,2$).\\
$\Psi_{n,j}(t)$&  the probability generating function of $N_{n,j}$.\\
$\Gamma(t_1,t_2)$& the probability generating function of $\left(N_{n,1},N_{n,2}\right).$
\end{tabular}
\begin{section}{Introduction}
In the context of binary systems, both the system and its components have only two states: working state  and failed state. The binary $k$-out-of-$n:G(F)$ system was introduced by Birnbaum et al. \cite{birnbaum61}. The system consists  of $n$ components; each component has two possible states, working or failed, and it functions (fails) if the total number of working (failed) components is at least $k$. The binary $k$-out-of-$n:G(F)$ system has been extensively studied by researchers such as in \cite{boland83proschan, sarje89, higashiyama2001, khatab2009, mahmoud2014ery}, and others. 
In the multi-state systems, the system and its components can be in $M+1$ possible states 0,1,...,M. The multi-state systems can model several real life systems that cannot be modeled by the binary systems.  

The first extension of the binary $k$-out-of-$n:G$ system to the case of multi-states is given by  El-Neweihi et al. \cite{elnewei1978}; the system is in state $j$ or above if at least $k$ components are in state $j$ or above, which means that the system has the same structure for every level of system states. Huang et al. \cite{huang2000generalized} proposed a definition of a generalized multi-state $k$-out-of-$n:G$ system as follows: The system is in state $j$ or above if there exists an integer value $l(j\leq l\leq M)$ such that at least $k_l$ components are in states at least as good as $l$.
In this definition, the $k_j$ does not have to be the same for different system states $j(1\leq j\leq M)$. This implies that the structure of the system may vary depending on the system-state levels. The generalized multi-state $k$-out-of-$n:G$ system can be used to model lines of products in plants \cite{huang2000generalized}, and power stations \cite{zuo2006tian}. Huang et al. \cite{huang2000generalized} considered two special cases of the definition of the multi-state $k$-out-of-$n:G$ system. When $k_1\leq k_2 \leq ...\leq k_M$, the system was called: increasing multi-state $k$-out-of-$n:G$ system, and when $k_1\geq k_2 \geq ...\geq k_M$ with at least one strict inequality, the system was called: decreasing multi-state $k$-out-of-$n:G$ system. Zou and Tian \cite{zuo2006tian} proposed a generalized multi-state $k$-out-of-$n:F$ system and gave the relationship between multi-state $k$-out-of-$n:G$ system and corresponding multi-state $k$-out-of-$n:F$ system.

In this paper, we consider a three-state $k$-out-of-$n:G$ system; both the system and its components can be in three possible states: state $2$ perfect functioning, state $1$ partial  working, and state $0$ complete failure. According to the definition given by  Huang et al. \cite{huang2000generalized}, a three-state $k$-out-of-$n:G$ system is in state $2$ if at least $k_2$ components are in state $2$, and the system is in state $1$ or above if at least $k_1$ components are in state $1$ or above, or at least $k_2$ components are in state $2$.  
Let $X_u$ be the random variable representing the state of component $u$ ($X_u=2$ if component $u$ is in perfect functioning, $X_u=1$ if component $u$ is in partially working, and $X_u=0$ if component $u$ is in complete failure), $u=1,2,...,n$. Assume that the system components are non-homogeneous Markov dependent. In other words, the state of  any component in the system depends only on the state of its preceding component and does not depend on the states of the other components. Mathematically 
\begin{eqnarray}
	Pr\left\{X_i=l_i|X_1=l_1,...,X_{i-1}=l_{i-1},X_{i+1}=l_{i+1},...,X_n=l_n\right\}&=&Pr\left\{X_i=l_i|X_{i-1}=l_{i-1}\right\}\nonumber \\
&=&p_i^{l_{i-1}l_i}
	\label{mdcpt}
\end{eqnarray}
for $i=1,2,...,n$ and $l_1,l_2,...,l_n \in \left\{0,1,2\right\}$.

We use the probability generating function method to evaluate the reliability of a three-state $k$-out-of-$n:G$ system consisting of non-homogeneous Markov dependent components.
This method has been used in many papers to study system reliability, marginal reliability importance, and joint reliability importance \cite{kamalja2012, kamalja2014comp, zhu2015joint, zhu2017lout, monarticle}. 
All these study were done in the case of binary systems. From what is in the literature and our knowledge, we think that the approach of the probability generating function was not used in the multi-state systems, which led us to invest and focus on this specific case and especially the three-state $k$-out-of-$n:G$ system.  

The main contribution of this work is to study the state distributions of a three-state $k$-out-of-$n:G$ system, which has not been studied before when the system components are non-homogeneous Markov dependent.

We structure the remainder of the paper as follows: in Section 2, we present a formula for computing the state distributions of an increasing three-state $k$-out-of-$n:G$ system and decreasing three-state $k$-out-of-$n:G$ system, respectively. In Section 3, we present two numerical examples to demonstrate the use of the formula. Finally, in Section 4, we give the conclusion of the paper. 
 
\end{section}
\setlength{\parindent}{0ex}
\begin{section}{Reliability evaluation}
\begin{subsection}{Increasing three-state $k$-out-of-$n:G$ system} 
\setlength{\parindent}{0ex}
Note that the increasing three-state $k$-out-of-$n:G$ system $(k_1\leq k_2)$ includes the constant three-state $k$-out-of-$n:G$ system $(k_1=k_2)$. In this case, the definition of a three-state $k$-out-of-$n:G$ system can be rephrased as follows: the system is in state $j$ or above iff at least $k_j$ components are in state $j$ or above.\\  
Let $X_u$ be the random variable representing the state of component $u, u=1,2,...,n$, where $X_u=2$ if component $u$ is in perfect functioning, $X_u=1$ if component $u$ is in partially working, and $X_u=0$ if component $u$ is in total failure. Define $N_{n,j}$ the total number of components that are in state $j$ or above ($j=1,2$). Then, the probability that an increasing three-state $k$-out-of-$n:G$ system is in state $j$ or above, $R^{j}(n)$, is
\begin{equation}
	R^j(n)=Pr\left\{N_{n,j}\geq k_j\right\}
	\label{Rj:inc}
\end{equation}
Let $\Psi_{n,j}(t)$ be the probability generating function of a distribution of $N_{n,j}$. Then
\begin{equation}
	\Psi_{n,j}(t)=E\left(t^{N_{n,j}}\right)=\sum_{x=0}^{n}{Pr\left\{N_{n,j}=x\right\}t^x}
	\label{pgf:inc}
\end{equation}
Thus, the probability that the system is in state $j$ or above can be obtained from the probability generating function, $\Psi_{n,j}(t)$, by the summation of coefficients of $t^x$ with $x\geq k_j$.\\
We first derive the expression of $\Psi_{n,j}(t)$ in Theorem \ref{th:inc}. Then, we use it to derive the formula of $R^j(n)$ for an increasing  three-state $k$-out-of-$n:G$ system in Proposition \ref{pro:inc}. 
\newtheorem{theo}{Theorem}
\begin{theo}
For an increasing three-state $k$-out-of-$n:G$ system with non-homogeneous Markov dependent components
\begin{equation*}
\Psi_{n,j}(t)=\bm{\bar{1}}\prod_{c=1}^{n}\bm{H}_c^j(t)\bm{1}, \; \; \mbox{for}\; j=1,2.
\end{equation*}
 Where $\bm{H}_c^j(t)$ is a $(3\times 3)$-matrix  for $c=1,2,...,n$ and $j=1,2$ as 
$$\bm{H}_c^j(t)=\left(\begin{array}{lll}
p_c^{00}  & p_c^{01}t^{\mathcal{I}_j(1)}  & p_c^{02}t^{\mathcal{I}_j(2)}  \\
p_c^{10}  & p_c^{11}t^{\mathcal{I}_j(1)}  & p_c^{12}t^{\mathcal{I}_j(2)}  \\
p_c^{20}  & p_c^{21}t^{\mathcal{I}_j(1)}  & p_c^{22}t^{\mathcal{I}_j(2)}  \\
\end{array}\right)$$
and $\bm{\bar{1}}=(1\;0\;0)$, $\bm{1}=(1\;1\;1)^{'}$.
\label{th:inc}
\end{theo}
\begin{proof}[\textbf{Proof}]
Consider an increasing three-state $k$-out-of-$n:G$ system. Assume that the system components are non-homogeneous Markov dependent. Let $\Psi_{n,j}(t)$ be the probability generating function of distribution of $N_{n,j}$ in a sequence of components whose states are $X_1,X_2,...,X_n$. Define $\Psi_{c,j}^{\alpha}(t)$ the probability generating function of the total number of components that are in state $j$ or above in $X_{c+1},X_{c+2},...,X_n$ given that $X_c=\alpha$, $\alpha=0,1,2$. Correspondingly, define column vector $\bm{\Psi}_{c,j}(t)=\left(\Psi_{c,j}^{0}(t),\Psi_{c,j}^{1}(t),\Psi_{c,j}^{2}(t)\right)^{'}$.\\
Assume that $Pr\left\{X_0=0\right\}=1$ and conditioning on $X_0$, we have
\begin{eqnarray}
	\Psi_{n,j}(t)&=&E\left(t^{N_{n,j}|X_0=0}\right)Pr\left\{X_0=0\right\}+E\left(t^{N_{n,j}|X_0\neq 0}\right)Pr\left\{X_0\neq 0\right\} \nonumber \\
	             &=&E\left(t^{N_{n,j}|X_0=0}\right) \nonumber \\
	             &=&\Psi_{0,j}^0(t) \nonumber \\
							 &=&\bm{\bar{1}}\bm{\Psi}_{0,j}(t)
							\label{psi:0}
\end{eqnarray}
Conditioning on the state of component $c$, for $c=1,2,...,n-1$, we obtain 
\begin{equation}
\Psi_{c-1,j}^\alpha(t)=\left\{\begin{array}{ll}
p_c^{00}\Psi_{c,j}^{0}(t)+p_c^{01}\Psi_{c,j}^{1}(t)t^{\mathcal{I}_j(1)}+p_c^{02}\Psi_{c,j}^{2}(t)t^{\mathcal{I}_j(2)}  & \mbox{for} \; \alpha=0 \\ \\
p_c^{10}\Psi_{c,j}^{0}(t)+p_c^{11}\Psi_{c,j}^{1}(t)t^{\mathcal{I}_j(1)}+p_c^{12}\Psi_{c,j}^{2}(t)t^{\mathcal{I}_j(2)}  & \mbox{for} \;\alpha=1 \\\\
p_c^{20}\Psi_{c,j}^{0}(t)+p_c^{21}\Psi_{c,j}^{1}(t)t^{\mathcal{I}_j(1)}+p_c^{22}\Psi_{c,j}^{2}(t)t^{\mathcal{I}_j(2)}  & \mbox{for} \; \alpha=2 
\end{array}\right.
\label{receqinc}
\end{equation}
The above relations in (\ref{receqinc}) for $c=1,2,...,n-1$ can be expressed  as 
\begin{equation}
\bm{\Psi}_{c-1,j}(t)=\bm{H}_c^j(t)\bm{\Psi}_{c,j}(t)	
\end{equation}
 Therefore, $\bm{\Psi}_{0,j}(t)=\left(\prod_{c=1}^{n}{\bm{H}_c^j(t)}\right)\bm{\Psi}_{n,j}(t)$. For $c=n$, we have $\Psi_{n,j}^{\alpha}(t)=1$ for $\alpha=0,1,2$, that is, $\bm{\Psi}_{n,j}(t)=\bm{1}$. Thus, by equation (\ref{psi:0}), the result follows.
\end{proof}
For a subset of components $S\subseteq\left\{1,2,...,n\right\}$, define matrix $\bm{G}_{u,S}^j$ as 
\begin{equation}
	G_{u,S}^j=\left\{\begin{array}{ll}
	\bm{H}_{u}^j(0)                 &  \mbox{if}\; u\notin S\\ \\
	\bm{H}_{u}^j(1)-\bm{H}_{u}^j(0) &  \mbox{if}\; u\in S
	\end{array}\right.
	\label{GuS:inc}
\end{equation}
\newtheorem{pro}{Proposition}
\begin{pro}
The probability that an increasing three-state $k$-out-of-$n:G$ system with non-homogeneous Markov dependent components is in state $j$ or above, $R^j(n)$, $j=1,2$ is
\begin{equation*}
	R^j(n)=\sum_{S:|S|\geq k_j}{\bm{\bar{1}}\prod_{u=1}^{n}{G_{u,S}^j}\bm{1}}
\end{equation*}
where the sum is taken over all the subsets whose sizes are greater than or equal $k_j$. 
\label{pro:inc}
\end{pro}
\begin{proof}[\textbf{Proof}]  
By equations (\ref{Rj:inc}) and (\ref{pgf:inc}), $R^j(n)$, is the sum of the coefficients of $t^x$ for $x\geq k_j$ in $\Psi_{n,j}(t)=\sum_{x=0}^{n}{Pr\left\{N_{n,j}=x\right\}t^x}.$\\
Note that $\bm{H}_c^j(t)=\bm{H}_c^j(0)+\left(\bm{H}_c^j(1)-\bm{H}_c^j(0)\right)t.$ By Theorem \ref{th:inc} and the definition of $\bm{G}_{u,S}^j$ in equation (\ref{GuS:inc}), the expression of $R^j(n)$ follows.
\end{proof}
After calculating $R^1(n)$ and $R^2(n)$. We can use the following equations to get the probability that the system is in state $j$, $r^j(n)$, $j=0,1,2.$
\begin{eqnarray*}
	r^2(n)&=&R^2(n)\\
	r^1(n)&=&R^1(n)-R^2(n)\\
	r^0(n)&=&1-R^1(n)\\
\end{eqnarray*}
\newtheorem{remarque}{Remark}
\begin{remarque}
If the system components are homogeneous Markov dependent, that is $p_u^{l_{u-1}l_u}=p^{l_{u-1}l_u}$ for all $u=1,2,...,n$. Then, matrix $\bm{H}_u^j(t)=\bm{H}^j(t)$ for all $u$. Thus, the system reliability is specified by Proposition \ref{pro:inc} for the case of homogeneous Markov dependent components.
\end{remarque}    

\end{subsection}
\begin{subsection}{Decreasing three-state $k$-out-of-$n:G$ system} 
In this case, the definition of a system can be rephrased as follows: the system is in state $2$ iff there are at least $k_2$ components  in state $2$, and the system is in state $1$ iff there are at least $k_1$ components in state $1$ or above and there are at most $k_2-1$ components  in state $2$.\\
The probability that a decreasing three-state $k$-out-of-$n:G$ system is in state $1$, $r^1(n)$, in terms of $N_{n,1}$ and $N_{n,2}$ is 
\begin{equation}
	r^1(n)=Pr\left\{N_{n,1}\geq k_1,N_{n,2}< k_2\right\}
	\label{r1:dec}
\end{equation}
The probability that a decreasing three-state $k$-out-of-$n:G$ system is in state $2$, $r^2(n)$, in terms of $N_{n,2}$ is
\begin{equation}
	r^2(n)=Pr\left\{N_{n,2}\geq k_2\right\}.
	\label{r2:dec}
\end{equation}
 Let $\Gamma(t_1,t_2)$ be the probability generating function of joint distribution of $\left(N_{n,1},N_{n,1}\right)$ in a sequence of components whose states are $X_1,X_2,...,X_n$. Then
\begin{equation}
	\Gamma(t_1,t_2)=E\left(t_1^{N_{n,1}}t_2^{N_{n,2}}\right)=\sum_{x=0}^{n}{}\sum_{y=0}^{n}{Pr\left\{N_{n,1}=x,N_{n,2}=y\right\}t_1^xt_2^y}
	\label{pgf:dec}
\end{equation}
Thus, the probability that a decreasing three-state $k$-out-of-$n:G$ system is in state $j$ can be obtained from the probability generating function $\Gamma(t_1,t_2)$, which is given by Theorem \ref{pgf2}.
\label{th:dec}
\begin{theo} \label{pgf2}
For a decreasing three-state $k$-out-of-$n:G$ system with non-homogeneous Markov dependent components
\begin{equation*}
\Gamma(t_1,t_2)=\bm{\bar{1}}\prod_{c=1}^{n}\bm{H}_c(t_1,t_2)\bm{1}
\end{equation*}
Where, $\bm{\bar{1}}=(1\;0\;0)$, $\bm{1}=(1\;1\;1)$, and $\bm{H}_c(t_1,t_2)$ is a $(3\times 3)$-matrix of order $3$ for $c=1,2,...,n$ as $$\bm{H}_c(t_1,t_2)=\left(\begin{array}{lll}
p_c^{00} & p_c^{01}t_1 & p_c^{02}t_1t_2 \\
p_c^{10} & p_c^{11}t_1 & p_c^{12}t_1t_2 \\
p_c^{20} & p_c^{21}t_1 & p_c^{22}t_1t_2 \\
\end{array}
\right)$$
\label{th:dec}
\end{theo}
\begin{proof}[\textbf{Proof}]
Consider a decreasing three-state $k$-out-of-$n:G$ system with non-homogeneous Markov dependent components. For integer $c=0,1,...,n$, let $\Phi_c^{\alpha}(t_1,t_2)$ be the probability generating function of the total number of components that are in state $1$ or above and the total number of components that are in state $2$ in $X_{c+1},X_{c+2},...,X_{n}$, given that $X_c=\alpha$, $\alpha=0,1,2$. Correspondingly define column victor $$\bm{\Phi}_c(t_1,t_2)=\left(\Phi_c^0(t_1,t_2),\; \Phi_c^1(t_1,t_2),\; \Phi_c^2(t_1,t_2) \right)^{'}$$
Assume that $Pr\left\{X_0=0\right\}=1$. Then
\begin{equation}
\Gamma(t_1,t_2)=\Phi_0^0(t_1,t_2)=\bm{\bar{1}}\bm{\Phi}_0(t_1,t_2)
\label{gamdec0}
\end{equation}
Conditioning on the state of component $c$, for $c = 1,2,...,n-1$, we obtain 
\begin{equation}
\Phi_{c-1}^\alpha(t)=\left\{\begin{array}{ll}
p_c^{00}\Phi_{c}^0(t_1,t_2)+p_c^{01}\Phi_{c}^{1}(t_1,t_2)t_1+p_c^{02}\Phi_{c}^2(t_1,t_2)t_1t_2  & \mbox{for} \; \alpha=0 \\ \\
p_c^{10}\Phi_{c}^0(t_1,t_2)+p_c^{11}\Phi_{c}^{1}(t_1,t_2)t_1+p_c^{12}\Phi_{c}^2(t_1,t_2)t_1t_2  & \mbox{for} \;\alpha=1 \\\\
p_c^{20}\Phi_{c}^0(t_1,t_2)+p_c^{21}\Phi_{c}^{1}(t_1,t_2)t_1+p_c^{22}\Phi_{c}^2(t_1,t_2)t_1t_2  & \mbox{for} \; \alpha=2 
\end{array}\right.
\label{receqdec}
\end{equation}
the above relations in (\ref{receqdec}) for $c=1,2,...,n-1$ can be expressed as
\begin{equation}
	\bm{\Phi}_{c-1}(t_1,t_2)=\bm{H}_c(t_1,t_2)\bm{\Phi}_c(t_1,t_2)
\end{equation}
Therefore, $\bm{\Phi}_0(t_1,t_2)=\left(\prod_{c=1}^{n}{\bm{H}_c(t_1,t_2)} \right) \bm{\Phi}_n(t_1,t_2)$. For $c=n$, we have $\Phi_n^{\alpha}(t_1,t_2)=1$ for $\alpha=0,1,2$, that is, $\bm{\Phi}_n(t_1,t_2)=\bm{1}$. Thus by equation (\ref{gamdec0}), the expression of $\Gamma(t_1,t_2)$ follows. 
\end{proof}
To derive a closed-form formula of $r^j(n)$, $j=1,2$, define matrix $\bm{G}_{u,S_1,S_2}$ for two subsets of components, $S_1$ and $S_2$ $\subseteq \left\{1,2,...,n\right\} $, as 
\begin{equation}
\bm{G}_{u,S_1,S_2}=\left\{\begin{array}{ll}
\bm{H}_u(0,0)                 &	\mbox{if}\; u\notin S_1 \;\mbox{and}\; u\notin S_2 \\
\bm{H}_u(1,0)-\bm{H}_u(0,0)   & \mbox{if}\; u\in S_1 \;\mbox{and}\; u\notin S_2 \\
\bm{0}_{3\time 3}            & \mbox{if}\; u\notin S_1 \;\mbox{and}\; u\in S_2 \\
\bm{H}_u(1,1)-\bm{H}_u(1,0)   & \mbox{if}\; u\in S_1 \;\mbox{and}\; u\in S_2 
\end{array}\right.
\label{GuSU:dec}
\end{equation}
\begin{pro}
The probability that a decreasing three-state $k$-out-of-$n : G$ system with non-homogeneous Markov dependent components is in state $j$, $r^j(n)$, $j= 0,1,2$ is
\begin{equation*}
	r^2(n)=\sum_{S_1:|S_1|\leq n}\sum_{S_2:|S_2|\geq k_2}{\bm{\bar{1}}\prod_{u=1}^{n}{\bm{G}_{u,S_1,S_2}}\bm{1}},
\end{equation*}
\begin{equation*}
	r^1(n)=\sum_{S_1:|S_1|\geq k_1}\sum_{S_2:|S_2|<k_2}{\bm{\bar{1}}\prod_{u=1}^{n}{\bm{G}_{u,S_1,S_2}}\bm{1}},
\end{equation*}
and
\begin{equation*}
r^0(n)=1-(r^1(n)+r^2(n))	
\end{equation*}

where the first summation in the first equation is taken over all the subsets of components $S_1$ and the second summation is taken over all the subsets of components $S_2$ whose sizes are greater than or equal $k_2$, and the first summation in the second equation is taken over all the subsets of components $S_1$ whose sizes are greater than or equal $k_1$, and the second summation is taken over all the subsets of components $S_2$ whose sizes are less than $k_2$.
\label{pro:dec} 
\end{pro}
\begin{proof}[\textbf{Proof}]
By equations (\ref{r1:dec}), (\ref{r2:dec}), and (\ref{pgf:dec}), $r^1(n)$ is the sum of coefficients of $t_1^xt_2^y$ for $x\geq k_1$ and $y<k_2$,  and $r^2(n)$ is the sum of coefficients of $t_1^xt_2^y$ for $x\geq 0$ and $y\geq k_2$ in $$\Gamma(t_1,t_2)=\sum_{x=0}^{n}{}\sum_{y=0}^{n}{Pr\left\{N_{n,1}=x,N_{n,2}=y\right\}t_1^xt_2^y}$$ 
Note that $\bm{H}_u(t_1,t_2)=\bm{H}_u(0,0)+\left(\bm{H}_u(1,0)-\bm{H}_u(0,0)\right)t_1+\left(\bm{H}_u(1,1)-\bm{H}_u(1,0)\right)t_1t_2.$ \\By Theorem \ref{th:dec} and the definition of $\bm{G}_{u,S,U}$ in equation (\ref{GuSU:dec}), the expression of $r^m(n)$, $m=1,2$ follows.
\end{proof}
\begin{remarque}
If the system components are homogeneous Markov dependent, that is $p_u^{l_{u-1}l_u}=p^{l_{u-1}l_u}$ for all $u=1,2,...,n$. Then, matrix $\bm{H}_u(t_1,t_2)=\bm{H}(t_1,t_2)$ for all $u$. Thus, the system reliability is specified by Proposition \ref{pro:dec} for the case of homogeneous Markov dependent components.
\end{remarque}
\end{subsection}
\end{section}
\begin{section}{Numerical examples}
\newtheorem{exemple}{Example} 
\begin{exemple}
Consider a three-state $k$-out-of-$3:G$ system with $k_1=2$ and $k_2=3$. Assume that the conditional probabilities of the three non-homogeneous Markov dependent components is\\ \\
$\left(p_1^{22},p_2^{22},p_3^{22}\right)=\left(0.6,\; 0.55, \; 0.55\right),\; \left(p_1^{12},p_2^{12},p_3^{12}\right)=\left(0.3, 0.25, 0.25\right),\; \mbox{and}\; \left(p_1^{02},p_2^{02},p_3^{02}\right)=$\\$\left(0.30, \; 0.35,\;0.25 \right)$ \\ 
$\left(p_1^{21},p_2^{21},p_3^{21}\right)=\left(0.3,\; 0.35, \; 0.3\right),\; \left(p_1^{11},p_2^{11},p_3^{11}\right)=\left(0.5,\; 0.5, \; 0.55\right), \; \mbox{and}\; \left(p_1^{01},p_2^{01},p_3^{01}\right)=$\\$\left(0.4,\; 0.45, \; 0.5\right)$\\ \\
Then, matrix $$\bm{H}_c^j(t)=\left(\begin{array}{lll}
p_c^{00}  & p_c^{01}t^{\mathcal{I}_j(1)}  & p_c^{02}t^{\mathcal{I}_j(2)}  \\
p_c^{10}  & p_c^{11}t^{\mathcal{I}_j(1)}  & p_c^{12}t^{\mathcal{I}_j(2)}  \\
p_c^{20}  & p_c^{21}t^{\mathcal{I}_j(1)}  & p_c^{22}t^{\mathcal{I}_j(2)}  \\
\end{array}\right)
,\; \mbox{for}\; c=1,2,3,\; \mbox{and} \; j=1,2.$$
By Theorem \ref{th:inc}, $\bm{\Psi}_{3,j}(t)=\bm{\bar{1}}\prod_{c=1}^{3}\bm{H}_{c}^j(t)\bm{1}$, $j=1,2$. Then
\begin{eqnarray*}
\bm{\Psi}_{3,1}(t)&=&	0.0050+0.63225t^3+0.29975t^2+0.06300t\\
\bm{\Psi}_{3,2}(t)&=&0.21750+0.18150t^3+0.27650t^2+0.32450t	
\end{eqnarray*}
Thus,
\begin{equation*}
	R^1(3)=0.29975+0.63225=0.93200
\end{equation*}
\begin{equation*}
	R^2(3)=0.18150
\end{equation*}
Alternatively, by Proposition \ref{pro:inc},
\begin{eqnarray*}
	R^1(3)&=&\sum_{S:2\leq|S|\leq 3}{\bm{\bar{1}}\prod_{u=1}^{3}{\bm{G}_{u,S}^1}\bm{1}}\\
	      &=&\sum_{S:|S|=2}{\bm{\bar{1}}\prod_{u=1}^{3}{\bm{G}_{u,S}^1}\bm{1}}+\sum_{S:|S|=3}{\bm{\bar{1}}\prod_{u=1}^{3}{\bm{G}_{u,S}^1}\bm{1}} \\
				&=&\sum_{i=1}^{3}{\prod_{u=1}^{3-i}{\bm{H}_u^1(0)\left(\bm{H}_i^1(1)-\bm{H}_i^1(0)\right)}\prod_{s=n-i+2}^{3}{\bm{H}_s^1(0)}}+\bm{\bar{1}}\prod_{u=1}^{3}{\left(\bm{H}_u^1(1)-\bm{H}_u^1(0)\right)}\bm{1}\\
				&=&0.29975+0.63225=0.93200\\
	R^2(3)&=&\sum_{S:|S|=3}{\bm{\bar{1}}\prod_{u=1}^{3}{\bm{G}_{u,S}^2}\bm{1}} \\
	      &=&\bm{\bar{1}}\prod_{u=1}^{3}{\left(\bm{H}_u^2(1)-\bm{H}_u^2(0)\right)}\bm{1}\\
				&=&0.18150
\end{eqnarray*}
Finally, 
\begin{eqnarray*}
r^2(3)&=&R^2(3)=0.18150\\
r^1(3)&=&R^1(3)-R^2(3)=0.93200-0.18150=0.75050\\
r^0(3)&=&1-R^1(3)=1-0.93200=0.06800
\end{eqnarray*}
\end{exemple}
\begin{exemple}
Consider a three-state $k$-out-of-$3:G$ system with $k_1=3$, and $k_2=1$. Assume that the conditional probabilities are the same as in Example $1$.\\
Then, the corresponding matrix $\bm{H}_c(t_1,t_2)$ is  
$$\bm{H}_c(t_1,t_2)=\left(\begin{array}{lll}
p_c^{00}  & p_c^{01}t_1  & p_c^{02}t_1t_2  \\
p_c^{10}  & p_c^{11}t_1  & p_c^{12}t_1t_2  \\
p_c^{20}  & p_c^{21}t_1  & p_c^{22}t_1t_2  \\
\end{array}\right)
,\; \mbox{for}\; c=1,2,3.$$
By Theorem \ref{th:dec}, $\Gamma(t_1,t_2)=\bm{\bar{1}}\prod_{c=1}^{3}\bm{H}_{c}(t_1,t_2)\bm{1}$. Then
\begin{eqnarray*}
\Gamma(t_1,t_2)&=&0.0050+0.18150t_2^3t_1^3+0.19275t_2^2t_1^3+0.17550t_2t_1^3+0.0825t_1^3+0.08375t_1^2t_2^2+\\
&&0.12375t_1^2t_2+0.09225t_1^2+0.03775t_1+0.02525t_1t_2
\end{eqnarray*}
Thus,
\begin{eqnarray*}
	r^2(3)&=&0.18150+0.1927+0.08375=0.45795.\\
	r^1(3)&=&0.17550+0.0825=0.25800.\\
	r^0(3)&=&1-(r^2(3)+r^1(3))=1-(0.45795+0.25800)=0.28405.
\end{eqnarray*}
Alternatively, by Proposition \ref{pro:dec}, for example for $j=1$
\begin{eqnarray*}
	r^1(3)&=&\sum_{S_1:|S_1|\geq 3}{}\sum_{S_2:|S_2|<2}{\bm{\bar{1}}\prod_{u=1}^{3}{\bm{G}_{u,S_1,S_2}}\bm{1}}\\
        &=&\sum_{S_1:|S_1|= 3}{}\sum_{S_2:|S_2|=1}{\bm{\bar{1}}\prod_{u=1}^{3}{\bm{G}_{u,S_1,S_2}}\bm{1}}+\sum_{S_1:|S_1|= 3}{}\sum_{S_2:|S_2|=0}{\bm{\bar{1}}\prod_{u=1}^{3}{\bm{G}_{u,S_1,S_2}}\bm{1}}\\
				&=&\sum_{i=1}^{3}{}\bm{\bar{1}}\prod_{u=1}^{i-1}{\left(\bm{H}_u(1,0)-\bm{H}_u(0,0)\right)}\left(\bm{H}_i(1,1)-\bm{H}_i(1,0)\right)\prod_{s=i+1}^{3}{\left(\bm{H}_s(1,0)-\bm{H}_s(0,0)\right)}\bm{1}\\
				& &+\bm{\bar{1}}\prod_{u=1}^{3}{\left(\bm{H}_u(1,0)-\bm{H}_u(0,0)\right)}\bm{1}\\
				&=&0.0825+0.17550\\
				&=&0.25800.			
\end{eqnarray*}

In Table 1, we present the state distributions of a three-state $k$-out-of-$n:G$ system for different values of $k_1,k_2$ and $n$, with components having the conditional probabilities:
\begin{center}
\begin{tabular}{llll}
$\left(p_c^{21},p_c^{22}\right)=\left(0.3,0.6\right)$&for $c=1,2,...,15$ &$\left(p_c^{21},p_c^{22}\right)=\left(0.3,0.65\right)$   &for $c=16,17,...,20$  \\
$\left(p_c^{11},p_c^{12}\right)=\left(0.5,0.35\right)$&for $c=1,2,...,15$ &$\left(p_c^{11},p_c^{12}\right)=\left(0.45,0.45\right)$  &for $c=16,17,...,20$\\
$\left(p_c^{01},p_c^{02}\right)=\left(0.45,0.3\right)$&for $c=1,2,...,5$ &$\left(p_c^{01},p_c^{02}\right)=\left(0.55,0.3\right)$  &for $c=6,7,...,15$ \\
$\left(p_c^{01},p_c^{02}\right)=\left(0.55,0.25\right)$&for $c=16,17,...,20$& &\\
\end{tabular}
\end{center}

\renewcommand{\tablename}{\textbf{Table}}
\begin{table}[h]
\begin{center}
\caption{State distributions of a three-state $k$-out-of-$n:G$ system for different values of $k_1,k_2$ and $n$}
\begin{tabular}{cccccccc}
\hline
\hline
$n$     & $k_1$& $k_2$& $r^0(n)$   & $r^1(n)$    & $r^2(n)$   & $R^1(n)$   & $R^2(n)$ \\
\hline
$10$    &  4   &  3   &0.0002071763&0.13342191280&0.8663709109&0.9997928237&0.8663709109\\
        &  5   &  3   &0.0013698082&0.13225928090&0.8663709109&0.9986301918&0.8663709109 \\
				&  6   &  4   &0.0084395255&0.26531485150&0.7262456230&0.9915604745&0.7262456230\\
				&  6   &  5   &0.0094690450&0.44387722540&0.5466537296&0.9905309550&0.5466537296\\
\hline			
$15$    &  5   &  4   &0.0000010609&0.07440229886&0.9255966402&0.9999989391&0.9255966402 \\
        &  7   &  5   &0.0000831322&0.15466683570&0.8452500321&0.9999168678&0.8452500321\\
				&  8   &  6   &0.0005412759&0.27127122260&0.7281875015&0.9994587241&0.7281875015 \\
				&  8   &  7   &0.0005575429&0.41640144220&0.5830410149&0.9994424571&0.5830410149 \\
\hline				
$20$    &  7   &  6   &0.0000000783&0.06870243220&0.9312974895&0.9999999217&0.9312974895\\
        &  9   &  7   &0.0000046993&0.13104703960&0.8689482611&0.9999953007&0.8689482611\\
				&  10  &  9   &0.0000293583&0.33736341170&0.6626072300&0.9999706417&0.6626072300\\
				&  12  &  10  &0.0007354415&0.46896212730&0.5303024312&0.9992645585&0.5303024312\\
				&  15  &  10  &0.0309837102&0.43871385880&0.5303024312&0.9690162900&0.5303024312\\
\hline
\hline
\end{tabular}
\end{center}
\end{table}

\end{exemple} 
\end{section} 
\begin{section}{Conclusion}
In this paper, we have studied the reliability of a three-state $k$-out-of-$n:G$ system with non-homogeneous Markov dependent components. We have developed  a closed-form formula to evaluate the state distributions of an increasing (decreasing) three-state $k$-out-of-$n:G$ system, respectively, using the probability generating function method. We have presented two numerical examples to demonstrate the use of the formula. Because every multi-state  $k$-out-of-$n:F$system has a corresponding multi-state $k$-out-of-$n:G$ system, all results obtained in this paper can be  easily extended to a three-state $k$-out-of-$n:F$ system.  As a future work, the multi-state $k$-out-of-$n:G$ system with components having more than three-state can be studied by the same methodology.
\end{section}

\end{document}